\documentclass{article}
\usepackage{amsmath,amssymb,amsthm,amscd,dsfont}
\numberwithin{equation}{section} \allowdisplaybreaks
\begin{document}
\newtheorem{theorem}{Theorem}[section]
\newtheorem{defin}{Definition}[section]
\newtheorem{prop}{Proposition}[section]
\newtheorem{corol}{Corollary}[section]
\newtheorem{lemma}{Lemma}[section]
\newtheorem{rem}{Remark}[section]
\newtheorem{example}{Example}[section]
%\label{} %\ref{}
\title{Generalized Sasaki metrics on tangent bundles}
\author{{\small by}\vspace{2mm}\\Izu Vaisman}
\date{}
\maketitle
{\def\thefootnote{*}\footnotetext[1]%
{{\it 2000 Mathematics Subject Classification: 53C15}.
\newline\indent{\it Key words and phrases}: Tangent bundle, Sasaki metric, Generalized geometry, Generalized metric.}}
\begin{center} \begin{minipage}{10cm}
A{\footnotesize BSTRACT. {We define a class of metrics that extend the Sasaki metric of a tangent manifold of a Riemannian manifold. The new metrics are obtained by the transfer of the generalized (pseudo-)Riemannian metrics of the pullback bundle $\pi^{-1}(TM\oplus T^*M)$, where $\pi:\mathcal{T}M\rightarrow M$ is the natural projection. We obtain the expression of the transferred metric and we define a canonical, metric connection with torsion. We calculate the torsion, curvature and Ricci curvature of this connection and give a few applications of the results. We also discuss the transfer of generalized complex and generalized K\"ahler structures from the pullback bundle to the tangent manifold.}}
\end{minipage}
\end{center} \vspace*{5mm}
%\noindent
%begin{center} %\section %\end{center}
\section{Introduction}
We use the typical notation of differential geometry, e.g., $\Gamma$ for spaces of global cross sections of bundles, $\chi$ for the space of vector fields, $\Omega$ for spaces of differential forms, etc. Everything will be $C^\infty$-smooth, except along the zero section of a vector bundle, where continuity may suffice. We assume that the reader is familiar with the geometry of a tangent manifold (everything needed below may be found in \cite{YI}). We also assume that the reader is familiar with Lie and Courant algebroids \cite{LWX} and with generalized structures \cite{Galt}.

Subsection 2.1 contains preliminaries on tangent manifolds $\mathcal{T}M$ (by $\mathcal{T}M$ we mean the total space of the tangent bundle $\pi:TM\rightarrow M$ of the manifold $M$) including Bott connections, tangent metrics and the definition of various identification isomorphisms between vector bundles over $\mathcal{T}M$. In Subsection 2.2 the basic results concerning generalized metrics and structures are transposed from the big tangent bundle $\mathbf{T}M=TM\oplus T^*M$ to the pullback bundle $\pi^{-1}(\mathbf{T}M)$. In Subsection 3.1 we define the generalized Sasaki metrics by transferring the generalized metrics of $\pi^{-1}(\mathbf{T}M)$ to $T\mathcal{T}M$. We give the expression of such a metric and define a corresponding, canonical, metric connection with torsion whose basic ingredient is a Cartan connection similar to that of Finsler geometry \cite{BCS}. We calculate the torsion, curvature and Ricci curvature of this connection and give some applications. In Subsection 3.2, motivated by the role of generalized geometry in string theory \cite{Zab}, we discuss the transfer of generalized complex and generalized K\"ahler structures from $\pi^{-1}(\mathbf{T}M)$ to the tangent manifold $\mathcal{T}M$.
\section{Preliminaries}
In this section we recall the terminology, notation and important definitions from the geometry of tangent bundles. Then, we recall basic results from generalized geometry, while transferring them to the pullback of the big tangent bundle of a given manifold to the tangent manifold.
\subsection{Geometry of tangent bundles}
Let $M$ be an $m$-dimensional manifold $M^m$. Then, on $\mathcal{T}M$, we have local coordinates
$(x^i,y^i)$ $(i=1,...,m)$, where $(x^i)$ are local coordinates on $M$ and $(y^i)$ are vector coordinates (velocities) with respect to the basis $\partial/\partial x^i$. The corresponding coordinate transformations are
\begin{equation}\label{transfcoordtg} \tilde{x}^i=\tilde{x}^i(x^j),
\,\tilde{y}^i=\frac{\partial\tilde{x}^i}{\partial x^j}y^j
\end{equation} (we use the Einstein summation convention). The fibers of $TM$ define the {\it vertical foliation} $x^i=const.$ and its tangent, {\it vertical bundle} $\mathcal{V}$.

We will denote by an upper index $v,c$ vertical lifts and complete lifts, respectively, and we recall the formulas \cite{YI}
$$
X^v=\xi^i\frac{\partial}{\partial y^i},\, X^c=\xi^i\frac{\partial}{\partial x^i}+y^j \frac{\partial\xi^i}{\partial x^j}\frac{\partial}{\partial y^i},$$ $$\alpha^v=\pi^*\alpha=\alpha_idx^i,\, \alpha^c=\alpha_idy^i+y^j\frac{\partial\alpha_i}{\partial x^j}dx^i,$$
where $X=\xi^i(\partial/\partial x^i),\, \alpha=\alpha_idx^i$.

The formula $S\mathcal{X}=(\pi_*\mathcal{X})^v$ ($\mathcal{X}\in T(\mathcal{T}M)$) defines a tensor field $S\in End(T\mathcal{T}M)$, called the {\it tangent structure} of $\mathcal{T}M$, with the properties $S^2=0,\,im\,S=ker\,S=\mathcal{V},\,\mathcal{N}_S=0$, where $\mathcal{N}_S$ is the Nijenhuis tensor
\begin{equation}\label{Nij} \mathcal{N}_S(\mathcal{X},\mathcal{Y})= [S\mathcal{X},S\mathcal{Y}]-S[S\mathcal{X},\mathcal{Y}] -S[\mathcal{X},S\mathcal{Y}]+S^2[\mathcal{X},\mathcal{Y}].
\end{equation}

Usually, it is convenient to choose a {\it horizontal bundle}, also called a
{\it nonlinear connection}, i.e., a complementary distribution $\mathcal{H}$ of $\mathcal{V}$ ($T\mathcal{T}M=\mathcal{H}\oplus\mathcal{V}$). A vector $X\in T_xM$ has a {\it horizontal lift} $X^h\in T_{(x,y)}\mathcal{T}M$ characterized by $X^h\in\mathcal{H}_{(x,y)},\pi_*X^h=X$. The horizontal lifts of $\partial/\partial x^i$ yield a local tangent basis of $\mathcal{T}M$ \begin{equation}\label{Hbases}X_i=\left(\frac{\partial}{\partial x^i}\right)^h= \frac{\partial}{\partial x^i}-t_i^j\frac{\partial}{\partial y^j},\;\frac{\partial}{\partial y^i}\end{equation} with the dual cotangent basis \begin{equation}\label{Hcobases}dx^i,\;\theta^i=dy^i+t^i_jdx^j,
\end{equation}
where $t_i^j$ are local functions on $\mathcal{T}M$, known as the coefficients of the nonlinear connection $\mathcal{H}$. By the first equality (\ref{Hbases}) a change of coordinates $\tilde{x}^i=\tilde{x}^i(x^j)$ implies $\tilde{X}_i=(\partial x^j/\partial\tilde{x}_i)X_j$ and, then, (\ref{transfcoordtg}) implies that the coefficients of the connection change by the rule
\begin{equation}\label{transft} \tilde{t}_i^j=\frac{\partial\tilde{x}^j}{\partial x^k}
\frac{\partial x^h}{\partial\tilde{x}^i}t_h^k-\frac{\partial x^h}{\partial\tilde{x}^i}
\frac{\partial^2\tilde{x}^j}{\partial x^h\partial x^l}y^l.\end{equation}

The notion of a horizontal lift has a natural extension to arbitrary tensors \cite{YI}. $S|_{\mathcal{H}}:\mathcal{H}\rightarrow\mathcal{V}$ is an isomorphism and we will denote by $S'$ the tensor equal to $(S|_{\mathcal{H}})^{-1}$ on $\mathcal{V}$ and to zero on $\mathcal{H}$. The tensor fields $P_{\mathcal{H}}=S+S'$ defines an almost paracomplex structure of $\mathcal{T}M$.
\begin{defin}\label{defBottc} {\rm
With a fixed choice of the horizontal bundle $\mathcal{H}$, a {\it Bott connection} on $\mathcal{T}M$ is a linear connection $\nabla$ that preserves the subbundles $\mathcal{H},\mathcal{V}$ and satisfies the conditions
\begin{equation}\label{BottV} \nabla_{\mathcal{X}}\mathcal{Y} =pr_{\mathcal{V}}[\mathcal{X},\mathcal{Y}],\,
\nabla_{\mathcal{Y}}\mathcal{X} =pr_{\mathcal{H}}[\mathcal{Y},\mathcal{X}].\end{equation}}\end{defin}

If $D$ is an arbitrary, linear connection, the addition of the derivatives
\begin{equation}\label{BottDV}
\nabla_{\mathcal{X}}\mathcal{X}'
=pr_{\mathcal{H}}D_{\mathcal{X}}\mathcal{X}',\, \nabla_{\mathcal{Y}}\mathcal{Y}'
=pr_{\mathcal{V}}D_{\mathcal{Y}}\mathcal{Y}'\; (\mathcal{X},\mathcal{X}'\in\mathcal{H},\,
\mathcal{Y},\mathcal{Y}'\in\mathcal{V})
\end{equation}
defines a connection $\nabla^D$, which we call the {\it Vr\u anceanu-Bott connection}.\footnote{In \cite{Bej0}, the author traced back the history of this connection to a 1931 paper by G. Vr\u anceanu \cite{Vr}.}
If $D$ has no torsion, $\nabla^D$ has the torsion
\begin{equation}\label{torsBott}
T_{\nabla^D}(\mathcal{Z}_1,\mathcal{Z}_2) =-pr_{\mathcal{V}}[pr_{\mathcal{H}}\mathcal{Z}_1, pr_{\mathcal{H}}\mathcal{Z}_2]= -R_{\mathcal{H}}(\mathcal{Z}_1,\mathcal{Z}_2),\end{equation} where $\mathcal{Z}_1,\mathcal{Z}_2\in\chi(\mathcal{T}M)$ and
$R_{\mathcal{H}}$ is the {\it Ehressmann curvature} of the non-linear connection $\mathcal{H}$.

Hereafter, the use of the term ``horizontal"  implies the fact that a horizontal bundle has been chosen and the terms ``vertical tensor", `` horizontal tensor" refer to tensors on the bundles $\mathcal{V},\mathcal{H}$, respectively. If necessary, we will look at vertical, respectively, horizontal tensors as tensors on $\mathcal{T}M$ by extending them by zero when evaluated on at least one horizontal, respectively, vertical argument. On the other hand, we can {\it reflect} (transfer) vertical tensors to horizontal tensors and vice-versa by $S$, $S'$, respectively. The reflection of $\tau$ will be denoted by $\tau^S$, if $\tau$ is horizontal, respectively, $\tau^{S'}$, if $\tau$ is vertical.
\begin{defin}\label{deftgm} {\rm
A {\it tangent metric} is a (pseudo-)Riemannian metric\footnote{``Pseudo" ia added if the metric is not positive definite.} $\gamma$ on $\mathcal{T}M$ with a non degenerate, restriction $\gamma_{\mathcal{V}}$, such that
$$ \gamma(S\mathcal{X},S\mathcal{X}')=\gamma(\mathcal{X},\mathcal{X}'),
\;\;\forall \mathcal{X},\mathcal{X}'\perp_\gamma\mathcal{V}.$$}\end{defin}

If $\gamma$ is a tangent metric, $\mathcal{H}=\mathcal{V}^{\perp_\gamma}$ is a horizontal bundle and $\gamma_{\mathcal{V}},g=\gamma_{\mathcal{H}}$ are a vertical, respectively horizontal, metric such that
$$\gamma_{\mathcal{V}}(S\mathcal{X},S\mathcal{X}')=g(\mathcal{X},\mathcal{X}').$$
The mapping $\gamma\mapsto(\mathcal{H},g)$ is a bijection between tangent metrics and pairs consisting of a horizontal bundle and a horizontal metric. The metric $g$ has the local expression
$$g=g_{ij}(x,y)dx^i\otimes dx^j,
\hspace{2mm}g_{ij}=g_{ji},\hspace{2mm}rank(g_{ij})=m.$$ For a given horizontal bundle $\mathcal{H}$, $\gamma$ is a corresponding tangent metric iff $\mathcal{H}\perp_\gamma\mathcal{V}$ and
\begin{equation}\label{tgPJ}
\gamma(P_{\mathcal{H}}\mathcal{Z},\mathcal{Z}')= \gamma(\mathcal{Z},P_{\mathcal{H}}\mathcal{Z}'),\;
\forall\mathcal{Z},\mathcal{Z}'\in T\mathcal{T}M.\end{equation}

In particular, the horizontal metric may be a {\it Lagrange metric} $g_{\mathcal{L}}$ with components $g_{ij}=\partial^2\mathcal{L}/\partial y^i\partial y^j$, where $\mathcal{L}$ is a function on $\mathcal{T}M$ such that $rank(g_{ij})=m$ (a {\it regular Lagrangian})\footnote{Sometimes, $\mathcal{L}$ itself is called a Lagrange metric} \cite{BM}. If $\mathcal{L}=\mathcal{F}^2$ where $\mathcal{F}$ is positive, positive homogeneous of degree $1$, and $g_{\mathcal{F}^2}$ is positive definite, then, $\mathcal{F}$ (and $g_{\mathcal{F}^2}$) is a {\it Finsler metric} \cite{BCS}.
We recall that a regular Lagrangian $\mathcal{L}$ defines a canonical horizontal bundle $\mathcal{H}_{\mathcal{L}}$ \cite{{BCS},{BM}}.

Let $\gamma$ be a tangent metric associated to a pair $(\mathcal{H},g)$ and $D$ the Levi-Civita connection of $\gamma$. The corresponding Vr\u anceanu-Bott connection $\nabla^D$ is called the {\it canonical connection} of $\gamma$.

The {\it Cartan tensor} of a horizontal metric (and of the corresponding tangent metric) is the horizontal tensor defined by
\begin{equation}\label{tensC2} C(\mathcal{X},\mathcal{X}',\mathcal{X}'')=\nabla^D_{S\mathcal{X}}g (\mathcal{X}',\mathcal{X}''),\;\mathcal{X},\mathcal{X}',\mathcal{X}''\in \mathcal{H}.\end{equation}
The local expression of $C$ is
$$ C=\frac{\partial g_{ij}}{\partial y^k}dx^k\otimes dx^i\otimes dx^j,$$
where $g_{ij}=g(X_i,X_j)$ with $X_i$ given by (\ref{Hbases}).
It follows that $C=0$ iff $g$ is the lift of a metric of $M$ and $C$ is totally symmetric iff each point has a neighborhood where $g$ is a Lagrange metric.

On the other hand, Finsler geometry suggests to introduce the {\it Cartan connection}
\begin{equation}\label{conexCartan} \nabla^{\mathcal{C}}_{\mathcal{Z}}\mathcal{Z}'=\nabla^D_{\mathcal{Z}}\mathcal{Z}' +\Psi(\mathcal{Z},\mathcal{Z}'),\end{equation}
where $\mathcal{Z},\mathcal{Z}',\mathcal{Z}''\in T\mathcal{T}M$ and $\Psi$ is defined by
$$ \gamma(\Psi(\mathcal{Z},\mathcal{Z}'), \mathcal{Z}'')=\frac{1}{2}C(S'\mathcal{Z},\mathcal{Z'},\mathcal{Z}'').$$
Therefore, $\Psi(\mathcal{Z},\mathcal{Z}')\in\mathcal{H}$ and $\Psi(\mathcal{Z},\mathcal{Z}')=0$ unless $\mathcal{Z}=\mathcal{Y}\in\mathcal{V},\mathcal{Z}'=\mathcal{X}\in\mathcal{H}$.
Using (\ref{tensC2}), it follows that $(\nabla^{\mathcal{C}}_{\mathcal{Z}}\gamma)(\mathcal{X},\mathcal{X}')=0$ for all $\mathcal{Z}\in T\mathcal{T}M,\mathcal{X},\mathcal{X}'\in\mathcal{H}$. The Cartan connection preserves the horizontal and vertical subbundles, but, it is not a Bott connection. The torsion of the Cartan connection is equal to that of $\nabla^D$ for two horizontal and two vertical arguments and $T^{\nabla^{\mathcal{C}}}(\mathcal{X},\mathcal{Y})= -\Psi(\mathcal{Y},\mathcal{X})$ if $\mathcal{X}\in\mathcal{H},\mathcal{Y}\in\mathcal{V}$.

We will be interested in the vector bundles $\pi^{-1}(TM)$ and $\pi^{-1}(\mathbf{T}M)$.
The mapping $$((x,y),X\in T_{x}M)\mapsto X^h\in\mathcal{H}_{(x,y)}$$ is a vector bundle isomorphism
$\pi^{-1}(TM)\approx\mathcal{H}$ with the inverse defined by $\pi_*$ and the mapping
\begin{equation}\label{frakh}
((x,y),(X,\alpha))\in \mathbf{T}_{x}M)\mapsto (X^h,\pi^*\alpha)
\in\mathcal{H}_{(x,y)}\oplus(ann\,\mathcal{V})
\end{equation}
is a vector bundle isomorphism
$\pi^{-1}(\mathbf{T}M)\approx\mathcal{H}\oplus\mathcal{H}^*$.
We will denote both mappings by the same symbol $\mathfrak{h}$.

The mapping $$((x,y),X\in T_{x}M)\mapsto X^v=SX^h$$ is an isomorphism
$\pi^{-1}(TM)\approx\mathcal{V}$ with the inverse defined by $\pi_*\circ S'$, and the mapping
$$((x,y),(X,\alpha)\in \mathbf{T}_{x}M)\mapsto(X^v=SX^h,[\alpha^c]_{ann\,\mathcal{V}}
=[(\pi^*\alpha )\circ S']_{ann\,\mathcal{V}})$$ is an isomorphism $\pi^{-1}(\mathbf{T}M)
\approx\mathcal{V}\oplus(T^*M/ann\,\mathcal{V})
\approx\mathcal{V}\oplus\mathcal{V}^*$.	 The inverse of this isomorphism has the following local expression
$$
(\eta^i(x,y)\frac{\partial}{\partial y^i}, \nu_i(x,y)[\theta^i]_{ann\,\mathcal{V}})
\mapsto((x,y),(\eta^i\frac{\partial}{\partial x^i},\nu_i(x,y)dx^i)),$$
where $\theta^i$ is defined by (\ref{Hcobases}).
We will denote the two mappings by the same symbol $\mathfrak{v}$ and notice the formulas
\begin{equation}\label{vh-1}
\mathfrak{v}\circ\mathfrak{h}^{-1}(\mathcal{X},\nu) =(S\mathcal{X},\nu\circ S'),\;
\mathfrak{h}\circ\mathfrak{v}^{-1}(\mathcal{Y}, \kappa)	 =(S'\mathcal{Y},\kappa\circ S),\end{equation}
where $\mathcal{X}\in\mathcal{H},\nu\in ann\,\mathcal{V},\mathcal{Y}\in\mathcal{V},\kappa\in ann\,\mathcal{H}$.

If $\mathcal{T}M$ is endowed with a tangent metric $\gamma$ with the corresponding horizontal bundle $\mathcal{H}\perp_\gamma\mathcal{V}$ and horizontal metric $g$, there are other interesting isomorphisms too. We indicate the whole picture in the following commutative {\it canonical identification diagram}
$$ \begin{array}{ccccc}
&&\mathcal{H}\oplus\mathcal{H}^*& \stackrel{\mathfrak{i}}{\longrightarrow}& T\mathcal{T}M\\ \pi^{-1}\mathbf{T}M &
\begin{array}{c}
\stackrel{\mathfrak{h}}{\nearrow}\vspace*{2mm}\\ \stackrel{\searrow}{\mathfrak{v}} \end{array}& \mathfrak{v}\mathfrak{h}^{-1}\downarrow\;
\uparrow\mathfrak{h}\mathfrak{v}^{-1}&& \downarrow\uparrow(S+S')  \\ && \mathcal{V}\oplus\mathcal{V}^*& \stackrel{\mathfrak{j}}{\longrightarrow}&T\mathcal{T}M \end{array}
$$
where, $\forall\mathcal{Z}\in T\mathcal{T}M$,
\begin{equation}\label{ij}
\mathfrak{i}^{-1}\mathcal{Z}= (pr_{\mathcal{H}}\mathcal{Z},\flat_\gamma S'pr_{\mathcal{V}}\mathcal{Z}),\; \mathfrak{j}^{-1}\mathcal{Z}= (pr_{\mathcal{V}}\mathcal{Z},\flat_\gamma Spr_{\mathcal{H}}\mathcal{Z}).\end{equation}
\subsection{Generalized geometry}
A {\it metric algebroid} \cite{VJMP} is a metric vector bundle $(E,g)$ with an anchor $\rho:E\rightarrow TM$ and a skew-symmetric bracket $[\,,\,]$ such that
\begin{equation}\label{metricbr}
(\rho e)(g(e_1,e_2))=g([e,e_1]+\partial(g(e,e_1)),e_2) +g(e_1,[e,e_2]+\partial(g(e,e_2))),\end{equation}
where $e,e_1,e_2\in\Gamma E$ and, for any function $f$, $\partial f$ is defined by $(\rho e)f=2g(\partial f,e)$. One can show that (\ref{metricbr}) has the consequence \cite{VJMP}
\begin{equation}\label{metriclin} [e_1,fe_2]=f[e_1,e_2]+(\rho e_1)(f)e_2-g(e_1,e_2)\partial f.\end{equation}
A {\it Courant algebroid} is a metric algebroid that also satisfies the axioms \cite{LWX}:
$$\rho[e_1,e_2]=[\rho e_1,\rho e_2],\, \rho\circ\partial=0,\; \sum_{Cycl}[[e_1,e_2],e_3]=(1/3)\partial\sum_{Cycl}g([e_1,e_2],e_3).$$

The basic example of a Courant algebroid is the {\it big tangent bundle} $\mathbf{T}M=TM\oplus T^*M$ with the anchor defined by $(X,\alpha)\mapsto X$. The structure group of $\mathbf{T}M$ has a natural reduction to $O(m,m)$ given by the neutral metric (i.e., of signature zero)
\begin{equation}\label{g}
\mathfrak{g}(\mathfrak{X},\mathfrak{Y})=\frac{1}{2}(\alpha(Y)+\mu(X)),
\end{equation}
where $\mathfrak{X}=(X,\alpha),\mathfrak{Y}=(Y,\mu)$ with $X,Y$ either vectors or vector fields and $\alpha,\mu$ either covectors or $1$-forms. The Courant bracket
of $\mathbf{T}M$ is $$
[\mathfrak{X},\mathfrak{Y}]=([X,Y],i(\mathcal{X})d\mu-i(\mathcal{Y})d\alpha -\frac{1}{2}d(\alpha(Y)-\mu(X)).$$

Furthermore, assume that $M$ is a {\it foliated manifold} with a foliation $\mathcal{F}$. For any choice of a horizontal bundle $\mathcal{H}$, $TM=\mathcal{H}\oplus T\mathcal{F}$, the bundle $\mathbf{T}M$ has the decomposition
$\mathbf{T}M=\mathbf{H}\oplus\mathbf{T}\mathcal{F},$ where $\mathbf{H}=\mathcal{H}\oplus\mathcal{H}^*$, $\mathbf{T}\mathcal{F}=T\mathcal{F}
\oplus T^*\mathcal{F}$ and $\mathcal{H}^*=ann(T\mathcal{F}), T^*\mathcal{F}=ann(\mathcal{H})$. Accordingly, the Courant bracket on $M$ induces the following brackets on $\mathbf{H}$ and $\mathbf{T}\mathcal{F}$:
\begin{equation}\label{Hbracket} \begin{array}{l}
[(X,\alpha),(X',\alpha')]_{\mathcal{H}}= pr_{\mathbf{H}}[(X,\alpha),(X',\alpha')]
\vspace*{2mm}\\ =(pr_{\mathcal{H}}[X,X'],i(X)d'\alpha'-i(X')d'\alpha -\frac{1}{2}d'(\alpha(X')-\alpha'(X))),\end{array}\end{equation}
where $X,X'\in\mathcal{H}, \alpha,\alpha'\in ann\,T\mathcal{F}$ and $d'=d|_{\mathcal{H}}$;
$$ \begin{array}{l}
[(Y,\nu),(Y',\nu')]_{\mathcal{F}}= pr_{\mathbf{T}\mathcal{F}}[(Y,\nu),(Y',\nu')]
\vspace*{2mm}\\ =([Y,Y'],i(Y)d''\nu'-i(Y')d''\nu -\frac{1}{2}d''(\nu(Y')-\nu'(Y))),\end{array}$$
where $Y,Y'\in T\mathcal{F}, \nu,\nu'\in ann\,\mathcal{H}$ and $d''$ is exterior differentiation along the leaves of $\mathcal{F}$.

Moreover, the bundles $\mathbf{H},\mathbf{T}\mathcal{F}$ have neutral metrics induced by $\mathfrak{g}$ and anchors
$\rho_{\mathcal{H}}(X,\alpha)=X,\rho_{\mathcal{F}}(Y,\nu)=Y$. The operations $(\mathfrak{g}|_{\mathbf{T}\mathcal{F}}, \rho_{\mathcal{F}},[\,,\,]_{\mathcal{F}})$ are the restrictions of the Courant algebroid operations of the big tangent bundle of the sum of the leaves of $\mathcal{F}$ to sections that are differentiable with respect to the original $C^\infty$-structure of $M$. Therefore, these operations make $\mathbf{T}\mathcal{F}$ into a Courant algebroid. On the other hand, the metric property (\ref{metricbr}) of the classical Courant bracket implies the same property for $(\mathfrak{g}|_{\mathbf{H}},\rho_{\mathcal{H}},[\,,\,]_{\mathcal{H}})$, therefore, $\mathbf{H}$ is a metric algebroid. But, unless $\mathcal{H}$ is integrable, $\mathbf{H}$ is not a Courant algebroid since $\rho_{\mathcal{H}}[(X,\alpha),(X'\alpha')]_{\mathcal{H}}\neq[X,X']$.

On a manifold $M$, generalized geometric structures are defined by reducing the structure group $O(m,m)$ of the big tangent bundle $\mathbf{T}M$ to a subgroup and integrability is defined by Courant bracket conditions (instead of Lie bracket conditions). The most studied were the {\it generalized complex structures}, alone and in conjunction with {\it generalized Riemannian metrics} and {\it generalized K\"ahler metrics} \cite{Galt}. These structures also have applications in string theory \cite{Zab}. Below, we present similar structures on the vector bundle $\pi^{-1}(\mathbf{T}M)\rightarrow\mathcal{T}M$. In other words, we refer to velocity dependent generalized (v.d.g.) geometry. \begin{defin}\label{defgenmetric} {\rm A
{\it v.d.g. (pseudo-)Riemannian metric} is a (pseudo-)Euclidian metric $\mathcal{G}$ on $\pi^{-1}(\mathbf{T}M)$, which has a non degenerate restriction to $\pi^{-1}(T^*M)$ and satisfies the $\mathfrak{g}$-compatibility condition
\begin{equation}\label{compparaH} \sharp_{\mathcal{G}}\circ\flat_\mathfrak{g}= \sharp_\mathfrak{g}\circ\flat_{\mathcal{G}}.
\end{equation}}\end{defin}

We shall briefly recall the following basic result \cite{Galt} (see also \cite{VT}), which also allows us to establish more of the required terminology.
\begin{prop}\label{sigmapsi} The v.d.g. metrics are in a bijective correspondence with pairs $(\sigma,\psi)$ where $\sigma$ is a non degenerate metric on $\pi^{-1}(TM)$ and $\psi$ is a $2$-form, $\psi\in\Gamma\wedge^2\pi^{-1}(T^*M)$.\end{prop}
\begin{proof} Define $$H=\sharp_{\mathcal{G}}\circ\flat_{\mathfrak{g}}=
\sharp_\mathfrak{g}\circ\flat_{\mathcal{G}}\in End(\pi^{-1}(\mathbf{T}M)),$$
equivalently,
\begin{equation}\label{compatGg}
\mathcal{G}((\mathcal{X},\alpha), (\mathcal{Y},\beta))=\mathfrak{g}(H(\mathcal{X},\alpha), (\mathcal{Y},\beta)),\end{equation}
where $\mathcal{X},\mathcal{Y}\in\pi^{-1}(TM), \alpha,\beta\in\pi^{-1}(T^*M)$.
The symmetry of $\mathcal{G}$ and condition (\ref{compparaH}) may be written as \begin{equation}\label{propluiH} \mathfrak{g}(H(\mathcal{X},\alpha), (\mathcal{Y},\beta))=\mathfrak{g}((\mathcal{X},\alpha), H(\mathcal{Y},\beta)),\;
H^2=Id.\end{equation}
Therefore, $H$ is a product structure on the bundle $\pi^{-1}(\mathbf{T}M)$ and	 $\pi^{-1}(\mathbf{T}M)=V_+\oplus V_-$ where the terms are the $\pm1$-eigenbundles of $H$. Furthermore, it was proven in \cite{VT} that the non-degeneracy of $\mathcal{G}|_{\pi^{-1}(T^*M)}$ is equivalent with $\pi^{-1}(T^*M)\cap V_\pm=0$.
Thus, we see that there exists a bijective correspondence between v.d.g. metrics $\mathcal{G}$ and the (necessarily para-complex) structures $H$ that satisfy (\ref{propluiH}) and the condition $\pi^{-1}(T^*M)\cap V_\pm=0$.

Such a structure $H$ may be represented under the following matrix form
\begin{equation}\label{H}
H\left(\begin{array}{c}\mathcal{X}\vspace{2mm}\\ \alpha \end{array}
\right)=\left(\begin{array}{cc} Q&\sharp_\sigma\vspace{2mm}\\
\flat_\beta&^t\hspace{-1pt}Q\end{array}\right) \left(
\begin{array}{c}\mathcal{X}\vspace{2mm}\\ \alpha \end{array}\right),
\end{equation} where the index $t$ denotes transposition and: 1) $\beta,\sigma\in\Gamma\odot^2(\pi^{-1}(T^*M))$, $\sigma$ is non degenerate, 2) $Q\in End(\pi^{-1}(TM))$ and the three tensors satisfy the equalities
\begin{equation}\label{relinH}
Q^2+\sharp_\sigma\circ\flat_\beta=Id,\; Q\circ\sharp_\sigma=-\sharp_\sigma\circ\hspace{1pt}^t\hspace{-1pt}Q,
\;\hspace{1pt}^t\hspace{-1pt}Q\circ\flat_\beta=-\flat_\beta\circ Q.
\end{equation}
This shows that the v.d.g. metrics correspond bijectively to the pairs $(\sigma,Q)$ and, if we define the form $\psi$ by the condition
$\flat_\psi=-\flat_\sigma\circ Q$, we are done.\end{proof}
\begin{rem}\label{obsiota} {\rm
The condition $\pi^{-1}(T^*M)\cap V_\pm=0$ implies that the projection $\pi^{-1}(\mathbf{T}M)\rightarrow\pi^{-1}(TM)$ restricts to isomorphisms between $V_\pm\approx\pi^{-1}(TM)$ and the inverse isomorphisms $\iota_\pm:\pi^{-1}(TM)\rightarrow V_\pm$ are given by $\iota_\pm(\mathcal{X})=(\mathcal{X},\flat_{\psi\pm\sigma}\mathcal{X})$
\cite{Galt}.}\end{rem}

A connection $\mathcal{D}$ on $\pi^{-1}(\mathbf{T}M)$ is a {\it paracomplex connection} if it commutes with the paracomplex structure $H$.
There exists a bijective correspondence between paracomplex connections $\mathcal{D}$ and pairs of connections $(D^+,D^-)$ on $\pi^{-1}(TM)$.
Indeed, $\mathcal{D}$ is paracomplex iff the parallel translations preserve the subbundles $V_\pm$. This is equivalent with the requirement that $\mathcal{D}$ be of the form
\begin{equation}\label{Dpm}
\mathcal{D}_{\mathcal{Z}}(\iota_\pm\mathcal{Y})
=\iota_\pm(D^{\pm}_{\mathfrak{Z}}\mathcal{Y})
\hspace{2mm} (\mathcal{Z}\in
\chi(\mathcal{T}M),\mathcal{Y}\in\Gamma\pi^{-1}(TM)),\end{equation}
where $D^{\pm}$ is the required pair.
Formula (\ref{Dpm}) implies the same relation for curvature, i.e.,
\begin{equation}\label{Rpm}
R_{\mathcal{D}}(\mathcal{Z},\mathcal{Z}')(\iota_\pm\mathcal{Y})
=\iota_\pm(R_{D^\pm}(\mathcal{Z},\mathcal{Z}')\mathcal{Y}).\end{equation}

Furthermore, a connection $\mathcal{D}$ on $\pi^{-1}(\mathbf{T}M)$ is a {\it double metric connection} if parallel translations preserve the metrics $\mathfrak{g}$ and $\mathcal{G}$,
and there exists a bijective correspondence between double metric connections $\mathcal{D}$ on $\pi^{-1}(\mathbf{T}M)$ and pairs $(D^+,D^-)$ of $\sigma$-preserving connections on $\pi^{-1}(TM)$.
Indeed, the connection $\mathcal{D}$ preserves the pair of metrics $(\mathfrak{g},\mathcal{G})$ iff it preserves the subbundles $V_\pm$ and their metrics $(\iota_\pm^{-1})^*\sigma$. Hence, $\mathcal{D}$ is a paracomplex connection of the form
(\ref{Dpm}) and connections $D^{\pm}$ must preserve the metric $\sigma$.
\begin{defin}\label{vdgcomplex} {\rm
A {\it v.d.g. almost complex structure} is an endomorphisms $\mathcal{J}$ of $\pi^{-1}(\mathbf{T}M)$ such that
$$\mathcal{J}^2=-Id,\;\mathfrak{g}((\mathcal{X},\alpha), \mathcal{J}(\mathcal{Y},\beta)) +\mathfrak{g}(\mathcal{J}(\mathcal{X},\alpha), (\mathcal{Y},\beta))=0.$$}\end{defin}

$\mathcal{J}$ may be represented by
\begin{equation}\label{J}
\mathcal{J}\left(\begin{array}{c}\mathcal{X}\vspace{2mm}\\ \alpha \end{array}
\right)=\left(\begin{array}{cc} A&\sharp_w\vspace{2mm}\\
\flat_b&-\hspace{1pt}^t\hspace{-1pt}A\end{array}\right) \left(
\begin{array}{c}\mathcal{X}\vspace{2mm}\\ \alpha \end{array}\right),
\end{equation}
$$b\in\Gamma\wedge^2\pi^{-1}(TM^*),\, w\in\Gamma\wedge^2\pi^{-1}(TM), \,A\in End\,\pi^{-1}(TM),$$ where
$$A^2+\sharp_w\circ\flat_b=-Id,\; A\circ\sharp_w=\sharp_w\circ\hspace{1pt}^t\hspace{-1pt}A,
\;\hspace{1pt}^t\hspace{-1pt}A\circ\flat_b=\flat_b\circ A.$$

A pair $(\mathcal{G},\mathcal{J})$ consisting of a v.d.g. (pseudo-)Riemannian metric and a v.d.g. almost complex structure on $\pi^{-1}(\mathbf{T}M)$ is {\it compatible} or a {\it v.d.g. almost (pseudo-)Hermitian structure} if
$H\circ\mathcal{J}=\mathcal{J}\circ H$. Then, $\mathcal{J}$ preserves $V_\pm$ and it induces complex structures in $V_\pm$. Therefore, a v.d.g. almost (pseudo-)Hermitian structure is equivalent with a quadruple $(\sigma,\psi,J_\pm)$ where $J_\pm$ are almost complex structures of $\pi^{-1}(TM)$, compatible with $\sigma$. If one replaces the pair $(J_+,J_-)$ by $(J_+,-J_-)$, one gets the {\it complementary structure} $\mathcal{J}'=H\circ\mathcal{J}$ \cite{Galt}.

In the usual case of $\mathbf{T}M$, a generalized almost complex structure is integrable iff the {\it Courant-Nijenhuis torsion}, defined by replacing the Lie brackets by Courant brackets in (\ref{Nij}), vanishes and a compatible pair $(\mathcal{G},\mathcal{J})$ is a generalized (pseudo-)K\"ahler structure if both $\mathcal{J}$ and $\mathcal{J}'$ are integrable.

In the v.d.g. case, we may define the {\it vertical version} of v.d.g. structures by transferring from $\pi^{-1}(\mathbf{T}M)$ to $\mathcal{V}\oplus\mathcal{V}^*$ via the isomorphism $\mathfrak{v}$. Then, we may define {\it vertical integrability} of the structure $\mathcal{J}$ by the vanishing of the vertical Courant-Nijenhuis torsion, constructed with the vertical Courant bracket. This is not very interesting since it reduces to the usual case along the sum of the vertical leaves with the supplementary requirement that differentiability be with respect to the original differential structure of $M$.
\begin{example}\label{excomplexvert} {\rm Let $J$ be a generalized almost complex structure of the manifold $M$, which has the corresponding matrix representation (\ref{J}) with entries $F\in End\,TM, \pi\in\Gamma\wedge^2TM,\varpi\in\Omega^2(M)$. Define corresponding vertical tensors
$$ \begin{array}{c} AX^v=(FX)^v,\;
w(\lambda^v\circ S',\mu^v\circ S')=\pi(\lambda,\mu),\;\vspace*{2mm}\\
b(X^v,Y^v)=\varpi^v(S'X^v,S'Y^v),\end{array}$$
where the index $v$ denotes vertical lift and $S'$ is defined using an arbitrary horizontal bundle.
The result is the vertical version of a v.d.g. almost complex structure $\mathcal{J}$ since the required algebraic conditions for $A,w,b$ follow from the corresponding conditions for $F,\pi,\varpi$. This structure $\mathcal{J}$ is the {\it vertical lift} of $J$. Using the fact that the local components of $A,w,b$ with respect to the basis $\partial/\partial y^i$ depend only on $x$, it is easy to check that $\mathcal{J}$ is always vertically integrable even if $J$ is not integrable on $M$.}\end{example}
\begin{example}\label{exKgen1} {\rm Let $(G,J)$ be a generalized almost Hermitian structure on $M$ and let $h$ be the almost-paracomplex structure that corresponds to $G$. We may define the vertical lift of $h$ as we did for $J$ in Example \ref{excomplexvert} and get a vertical generalized metric $\mathcal{G}$. We also have the vertically integrable lift $\mathcal{J}$ of $J$. The algebraic relations satisfied by $G,J$ hold $\mathcal{G},\mathcal{J}$ too, and the complementary structure $\mathcal{J}'$ will be the vertical lift of $J^c$. By the conclusion of Example \ref{excomplexvert}, $\mathcal{J}'$ is vertically integrable too and $(\mathcal{G},\mathcal{J})$ is a vertical, generalized K\"ahler structure.}\end{example}
\section{Generalized metrics on tangent manifolds}
Velocity dependent generalized structures become of interest for the geometry of tangent manifolds if they are transferred to $T\mathcal{T}M$ by the intermediary of $\mathbf{H}=\mathcal{H}\oplus\mathcal{H}^*$.
\subsection{Generalized Sasaki metrics}
Let $\mathcal{G}$ be a v.d.g. (pseudo-)Riemannian metric on $M$, associated with the para-complex structure $H$ represented by (\ref{H}).
Then, $\mathfrak{h}$ sends the related tensors $\mathcal{G},\sigma,\psi,\beta,Q$ of Section 2.2 to horizontal tensors, which we will denote by the same symbols. The pair $(\mathcal{H},\sigma)$ produces a {\it subordinated tangent metric} $\gamma$ on $\mathcal{T}M$, which, in turn, yields a canonical identification $\mathfrak{i}$.
\begin{defin}\label{defgenSas} {\rm A (pseudo-)Riemannian metric $\mathcal{G}$ of the tangent manifold $\mathcal{T}M$ defined by the $\mathfrak{i}\circ\mathfrak{h}$-transfer of a v.d.g. metric $\mathcal{G}$ of $\pi^{-1}(\mathbf{T}M)$ will be called a {\it generalized Sasaki metric}.}\end{defin}
\begin{prop} The generalized Sasaki metric $\mathcal{G}$ is given by the following values on horizontal and vertical arguments:
\begin{equation}\label{Grondcugamma}
2\mathcal{G}(\mathcal{Z},\mathcal{Z}')=\left\{
\begin{array}{ll} \hspace*{3mm}\beta(\mathcal{Z},\mathcal{Z}')&\,{\rm if}\,\mathcal{Z},\mathcal{Z}'
\in\mathcal{H},\vspace*{2mm}\\ -\psi(\mathcal{Z},S'\mathcal{Z}')&\,{\rm if}\,
\mathcal{Z}\in\mathcal{H},\mathcal{Z}'\in\mathcal{V},\vspace*{2mm}\\ \hspace*{3mm}\sigma(S'\mathcal{Z},S'\mathcal{Z}')&\,{\rm if}\, \mathcal{Z},\mathcal{Z}'\in\mathcal{V}.
\end{array}\right.\end{equation} \end{prop}
\begin{proof} Using formula (\ref{ij}), it follows that the metric $\mathfrak{g}|_{\mathbf{H}}$ defined by (\ref{g}) transfers to the metric
\begin{equation}\label{gF} \mathfrak{g}_{\mathbf{H}}(\mathcal{Z},\mathfrak{Z}')=\frac{1}{2} [\sigma(S'pr_{\mathcal{V}}\mathcal{Z},pr_{\mathcal{H}}\mathcal{Z}')+
\sigma(S'pr_{\mathcal{V}}\mathcal{Z}',pr_{\mathcal{H}}\mathcal{Z})],
\end{equation} where $\mathcal{Z},\mathcal{Z}'\in T\mathcal{T}M$.
(This metric is compatible with the almost paracomplex structure $F$ defined by $F|_{\mathcal{H}}=Id,F|_{\mathcal{V}}=-Id$, i.e., $(\mathfrak{g},F)$ is an almost para-Hermitian structure.)

From the relations (\ref{relinH}) we get
\begin{equation}\label{betacuQ}\flat_\beta=
\flat_\sigma\circ(Id-Q^2)=\flat_\gamma\circ(Id-Q^2).\end{equation} Then, using
(\ref{ij}) again, we see that the identification $\mathfrak{i}$ defined by $\gamma$ sends $H$ to the almost paracomplex structure of $\mathcal{T}M$ given by
\begin{equation}\label{HZ} H\mathcal{Z}=(Q+S\circ(Id-Q^2)) pr_{\mathcal{H}}\mathcal{Z}+(S'-SQS')pr_{\mathcal{V}}\mathcal{Z},
\;\mathcal{Z}\in\chi(\mathcal{T}M).\end{equation}
Now, the required result follows from (\ref{HZ}) and (\ref{compatGg}).\end{proof}
\begin{prop}\label{SasgencuQfrak} A generalized Sasaki metric is equivalent with a pair $(\gamma,\Phi)$ where $\gamma$ is a tangent metric, $\Phi$ is a $\gamma$-symmetric endomorphism of $T\mathcal{T}M$ such that $\Phi\circ P_{\mathcal{H}}\circ\Phi=P_{\mathcal{H}}$.\end{prop}
\begin{proof}
For a generalized Sasaki metric, formula (\ref{gF}) and the definition $P_{\mathcal{H}}=S+S'$ imply the equivalence between the first equality (\ref{propluiH}) and the equality
\begin{equation}\label{GgammaPH}
\gamma((P_{\mathcal{H}}\circ H)\mathcal{Z},\mathcal{Z}')= \gamma(\mathcal{Z},(P_{\mathcal{H}}\circ H)\mathcal{Z}').\end{equation}
This shows that the generalized Sasaki metrics are in a bijective correspondence with pairs $(\gamma,H)$ where $\gamma$ is a tangent metric, $H$ is an almost product structure and (\ref{GgammaPH}) holds. Therefore, if we take $\Phi=P_{\mathcal{H}}\circ H$, the required conclusion follows because $P_{\mathcal{H}}^2=Id.$ \end{proof}

Notice also the following consequence of formulas (\ref{gF}), (\ref{compatGg}) and (\ref{tgPJ}): $$\mathcal{G}(\mathcal{Z},\mathcal{Z}') =\frac{1}{2}\gamma(H\mathcal{Z},P_{\mathcal{H}}\mathcal{Z}') =\frac{1}{2}\gamma((P_{\mathcal{H}}\circ H)\mathcal{Z},\mathcal{Z}').$$

If $\psi=0$, we have $Q=0,H=P_{\mathcal{H}},\Phi=Id$. Furthermore, if $\sigma$ is defined by a metric of $M$ and if $\mathcal{H}$ is the horizontal bundle of the Levi-Civita connection of $\sigma$, then, (\ref{Grondcugamma}) is half the usual Sasaki metric.

If $\mathcal{G}$ is a generalized Sasaki metric and $H$ is the corresponding para-Hermitian structure, the pair $(\mathcal{G},H)$ is a {\it (pseudo-)Riemannian almost para-complex-product structure}, i.e., it satisfies the	 condition $\mathcal{G}(H\mathcal{Z},\mathcal{Z}')= \mathcal{G}(\mathcal{Z},H\mathcal{Z}')$ and $H$ is almost para-complex. We will consider linear connections of $\mathcal{T}M$ that preserve the structure $(\mathcal{G},H)$. Since this is equivalent with the preservation of $(\mathcal{G},\mathfrak{g}_{\mathbf{H}})$, the connections are {\it double metric connections}.
\begin{prop}\label{propexprconex} Let $\mathcal{D}$ be a double metric connection on $\pi^{-1}(\mathbf{T}M)\approx\mathbf{H}$ with the corresponding pair $D^\pm$ of connections on $\mathcal{H}$. Define the connection $D^0=(D^++D^-)/2$ and put
$$D^\pm_{\mathcal{Z}}\mathcal{X}=D^0_{\mathcal{Z}}\mathcal{X}
\pm\lambda_{\mathcal{Z}}\mathcal{X},$$
where $\mathcal{Z}\in T\mathcal{T}M,\mathcal{X},\mathcal{X}'\in\mathcal{H}$
and $\sigma(\lambda_{\mathcal{Z}}\mathcal{X},\mathcal{X}')
+\sigma(\mathcal{X},\lambda_{\mathcal{Z}}\mathcal{X}')=0$.
Then, $\mathfrak{i}$ transfers the connection $\mathcal{D}$ to the following linear connection of the manifold $\mathcal{T}M$
\begin{equation}\label{cazconexgen} \begin{array}{l}
\mathcal{D}_{\mathcal{Z}}(S\mathcal{X}) =SD^0_{\mathcal{Z}}\mathcal{X}+(Id-SQ) \lambda_{\mathcal{Z}}\mathcal{X},\vspace*{2mm}\\ \mathcal{D}_{\mathcal{Z}}\mathcal{X}=
D^0_{\mathcal{Z}}\mathcal{X}+\lambda_{\mathcal{Z}}(Q\mathcal{X}) +S[(D^0_{\mathcal{Z}}Q)\mathcal{X}+\lambda_{\mathcal{Z}}\mathcal{X} -Q\lambda_{\mathcal{Z}}(Q\mathcal{X})].
\end{array}\end{equation}\end{prop}
\begin{proof} The pair	$D^\pm$ was defined in Section 2.2. Calculating the mapping $\mathfrak{i}\circ\mathfrak{h}\circ\iota_\pm$ we get the horizontal version of $\iota_\pm$:
$$ \iota_\pm\mathcal{X}=\mathcal{X}+ S\sharp_\sigma\flat_{\psi\pm\sigma}\mathcal{X}= \mathcal{X}-SQ\mathcal{X}\pm S\mathcal{X},\; \mathcal{X}\in\mathcal{H},$$
whence
\begin{equation}\label{auxiotapm}
\iota_+\mathcal{X}-\iota_-\mathcal{X}=2S\mathcal{X},\;
\iota_+\mathcal{X}+\iota_-\mathcal{X}=2(\mathcal{X}-SQ\mathcal{X}).
\end{equation}
If we solve (\ref{auxiotapm}) for $\mathcal{X},S\mathcal{X}$, calculate $\mathcal{D}_{\mathcal{Z}}$ with (\ref{Dpm}) and use again (\ref{auxiotapm}) for the results, we get (\ref{cazconexgen}).\end{proof}

If $D^+=D^-$, $\lambda=0$ and the connection $\mathcal{D}$ preserves the vertical bundle. In the definition below we introduce the simplest canonical double metric connection.
\begin{defin}\label{defdmCartan} {\rm The double metric connection defined by the pair $D^+=D^-=\nabla^{\mathcal{C}}|_{\mathcal{H}}$, where $\nabla^{\mathcal{C}}$ is the Cartan connection (\ref{conexCartan}), will be called the {\it double metric Cartan connection} and it will be denoted by $\mathcal{D}^{\mathcal{C}}$.}\end{defin}

Using formulas (\ref{cazconexgen}) and (\ref{conexCartan}), we get
\begin{equation}\label{dmC} \begin{array}{l}
\mathcal{D}^{\mathcal{C}}_{\mathcal{X}}\mathcal{X}'=
\nabla^{D}_{\mathcal{X}}\mathcal{X}' +S((\nabla^{D}_{\mathcal{X}}Q)\mathcal{X}'), \,\mathcal{D}^{\mathcal{C}}_{\mathcal{X}}\mathcal{Y}= S\nabla^{D}_{\mathcal{X}}(S'\mathcal{Y}),\vspace*{2mm}\\ \mathcal{D}^{\mathcal{C}}_{\mathcal{Y}}\mathcal{X}=
\nabla^{D}_{\mathcal{Y}}\mathcal{X} +\Psi(\mathcal{Y},\mathcal{X}) +S[(\nabla^{D}_{\mathcal{Y}}Q)\mathcal{X}
+\Psi(\mathcal{Y},Q\mathcal{X})-Q\Psi(\mathcal{Y},\mathcal{X})],
\vspace*{2mm}\\ \mathcal{D}^{\mathcal{C}}_{\mathcal{Y}}\mathcal{Y}'= S[\nabla^{D}_{\mathcal{Y}}(S'\mathcal{Y}')+\Psi(\mathcal{Y},S'\mathcal{Y}')],
\end{array}
\end{equation}
where $\mathcal{X},\mathcal{X}'\in\mathcal{H},\mathcal{Y},\mathcal{Y}'\in\mathcal{V}$.

Following are the results of the standard computation of the torsion and curvature of connection $\mathcal{D}^{\mathcal{C}}$ and some consequences.

From (\ref{dmC}) and (\ref{torsBott}), we get
\begin{equation}\label{torsCartan} \begin{array}{l}
T_{\mathcal{D}^{\mathcal{C}}}(\mathcal{X},\mathcal{X}') = -pr_{\mathcal{V}} [\mathcal{X},\mathcal{X}'] - S[\nabla^D_{\mathcal{X}}(Q\mathcal{X}') -\nabla^D_{\mathcal{X}'}(Q\mathcal{X}) \vspace*{2mm}\\ \hspace*{15mm}-Qpr_{\mathcal{H}}[\mathcal{X},\mathcal{X}']]
,\vspace*{2mm}\\ T_{{\mathcal{D}}^{\mathcal{C}}}(\mathcal{Y},\mathcal{Y}') =
S[\Psi(\mathcal{Y},S'\mathcal{Y}')-\Psi(\mathcal{Y}',S'\mathcal{Y})],\vspace*{2mm}\\ T_{\mathcal{D}^{\mathcal{C}}}(\mathcal{X},\mathcal{Y}) =-\nabla^D_{\mathcal{X}}\mathcal{Y}-\Psi(\mathcal{Y},\mathcal{X})
+S[\nabla^D_{\mathcal{X}}(S'\mathcal{Y})- (\nabla^D_{\mathcal{Y}}Q)\mathcal{X}
\vspace*{2mm}\\ \hspace*{15mm}-\Psi(\mathcal{Y},Q\mathcal{X}) +Q\Psi(\mathcal{Y},\mathcal{X})],
\end{array}\end{equation}
where the second formula holds at every point and it holds for vector fields $\mathcal{Y},\mathcal{Y}'$ if the fields $S'\mathcal{Y},S'\mathcal{Y}'$ are assumed to be projectable to $M$. For arbitrary, vertical vector fields, we have
$$T_{{\mathcal{D}}^{\mathcal{C}}}(\mathcal{Y},\mathcal{Y}') = \nabla^D_{\mathcal{Y}}(S'\mathcal{Y}') -\nabla^D_{\mathcal{Y}'}(S'\mathcal{Y})-[\mathcal{Y},\mathcal{Y}'] +S[\Psi(\mathcal{Y},S'\mathcal{Y}')-\Psi(\mathcal{Y}',S'\mathcal{Y})].$$ If $S'\mathcal{Y},S'\mathcal{Y}'$ are projectable, the first two terms vanish and so does the third, since the vanishing of the Nijenhuis tensor (\ref{Nij}) of $S$ yields
$$ [\mathcal{Y},\mathcal{Y}']=[S(S'\mathcal{Y}),S(S'\mathcal{Y}')]= S([\mathcal{Y},S'\mathcal{Y}']
+[S'\mathcal{Y},\mathcal{Y}'])=0$$
($S$ is applied to a vertical field).

Formulas (\ref{torsCartan}) lead to an expression of the Levi-Civita connection $\mathcal{D}^{\mathcal{G}}$ of a generalized Sasaki metric $\mathcal{G}$:
\begin{equation}\label{exprLCcuCartan}
D^{\mathcal{G}}_{\mathcal{Z}}\mathcal{Z}'= \mathcal{D}^{\mathcal{C}}_{\mathcal{Z}}\mathcal{Z}' -\frac{1}{2}T_{\mathcal{D}^{\mathcal{C}}}(\mathcal{Z},\mathcal{Z}') +\frac{1}{2}\Xi(\mathcal{Z},\mathcal{Z}'),\end{equation}
where $\mathcal{Z},\mathcal{Z}'\in T\mathcal{TM}$ and $\Xi$ is defined by
$$\mathcal{G}(\Xi(\mathcal{Z},\mathcal{Z}'),\mathcal{Z}'') =\mathcal{G}(\mathcal{Z}',T_{\mathcal{D}^{\mathcal{C}}} (\mathcal{Z},\mathcal{Z}'')) +\mathcal{G}(\mathcal{Z},T_{\mathcal{D}^{\mathcal{C}}} (\mathcal{Z}',\mathcal{Z}'')).$$ (Check that the connection defined by the right hand side of (\ref{exprLCcuCartan}) preserves $\mathcal{G}$ and has no torsion.)

In view of (\ref{Rpm}), the calculations that gave formulas (\ref{cazconexgen}) hold for the curvature operator too and we get
\begin{equation}\label{curbCartan} \begin{array}{l}
R_{\mathcal{D}^{\mathcal{C}}}(\mathcal{Z},\mathcal{Z}')(\mathcal{X}) =
R_{\nabla^{\mathcal{C}}}(\mathcal{Z},\mathcal{Z}')(\mathcal{X}) +S[R_{\nabla^{\mathcal{C}}}(\mathcal{Z},\mathcal{Z}')(Q\mathcal{X})\vspace*{2mm}\\	 \hspace*{15mm}-QR_{\nabla^{\mathcal{C}}}(\mathcal{Z},\mathcal{Z}')(\mathcal{X})], \vspace*{2mm}\\ R_{\mathcal{D}^{\mathcal{C}}}(\mathcal{Z},\mathcal{Z}')(S\mathcal{X}) =
SR_{\nabla^{\mathcal{C}}}(\mathcal{Z},\mathcal{Z}')(\mathcal{X}).
\end{array}\end{equation}

We may also express $R_{\mathcal{D}^{\mathcal{C}}}$ by means of $R_{\nabla^D}$, where $\nabla^D$ is the canonical connection of the subordinated tangent metric $\gamma$ by using the known relation between the curvature of two connections, which is
\begin{equation}\label{twocurvatures}
\begin{array}{l}
R_{\nabla^{\mathcal{C}}}(\mathcal{Z},\mathcal{Z}')\mathcal{Z}''= R_{\nabla^D}(\mathcal{Z},\mathcal{Z}')\mathcal{Z}'' +\nabla^D_{\mathcal{Z}}(\Psi(\mathcal{Z}',\mathcal{Z}'')) -\nabla^D_{\mathcal{Z}'}(\Psi(\mathcal{Z},\mathcal{Z}''))\vspace*{2mm}\\ \hspace*{23mm}\,+\Psi(\mathcal{Z},\nabla^D_{\mathcal{Z}'}\mathcal{Z}'')- \Psi(\mathcal{Z}',\nabla^D_{\mathcal{Z}}\mathcal{Z}'') -\Psi([\mathcal{Z},\mathcal{Z}'],\mathcal{Z}'')
 \vspace*{2mm}\\ \hspace*{23mm}\,+\Psi(\mathcal{Z},\Psi(\mathcal{Z}',\mathcal{Z}''))
-\Psi(\mathcal{Z}',\Psi(\mathcal{Z},\mathcal{Z}'')).\end{array}\end{equation}
The full expression of (\ref{twocurvatures}) occurs if $\mathcal{Z},\mathcal{Z}'\in\mathcal{V},\mathcal{Z}''\in\mathcal{H}$. In the other cases, if we denote by $\mathcal{X},\mathcal{Y}$ horizontal and vertical arguments, respectively, since $\Psi$ vanishes unless its arguments are $(\mathcal{Y},\mathcal{X})$, we have
\begin{equation}\label{curbCD0}
\begin{array}{l}
R_{\nabla^{\mathcal{C}}}(\mathcal{Z},\mathcal{Z}')\mathcal{Y}= R_{\nabla^D}(\mathcal{Z},\mathcal{Z}')\mathcal{Y},\vspace*{2mm}\\
R_{\nabla^{\mathcal{C}}}(\mathcal{X},\mathcal{X}')\mathcal{X}''= R_{\nabla^D}(\mathcal{X},\mathcal{X}')\mathcal{X}'' -\Psi(pr_{\mathcal{V}}[\mathcal{X},\mathcal{X}'],\mathcal{X}''),\vspace*{2mm}\\
R_{\nabla^{\mathcal{C}}}(\mathcal{X},\mathcal{Y})\mathcal{X}'= R_{\nabla^D}(\mathcal{X},\mathcal{Y})\mathcal{X}' +\nabla^D_{\mathcal{X}}(\Psi(\mathcal{Y},\mathcal{X}'))\vspace*{2mm}\\
\hspace*{23mm}-\Psi(\mathcal{Y},\nabla^D_{\mathcal{X}}\mathcal{X}')
-\Psi(pr_{\mathcal{V}}[\mathcal{X},\mathcal{Y}],\mathcal{X}').
\end{array}\end{equation}
In all the cases, the difference between the two curvatures is horizontal.
\begin{prop}\label{Cartan=LC} Let $\mathcal{G}$ be a generalized Sasaki metric of $\mathcal{T}M$ defined by the horizontal lift of a pair $(s,\zeta)$, where $s$ is a (pseudo-)Riemannian metric of $M$ and $\zeta\in\Omega^2(M)$ is a $2$-form that is parallel with respect to the Levi-Civita connection of $s$. Then, the double metric Cartan connection $\mathcal{D}^{\mathcal{C}}$ of $\mathcal{G}$ is equal to its Levi-Civita connection $D^{\mathcal{G}}$ iff the metric $s$ is flat and $\mathcal{H}$ is the horizontal bundle of the Levi-Civita connection of $s$.\end{prop}
\begin{proof} Consider the horizontal lifts $\sigma=s^h,\psi=\zeta^h$.
Since $\sigma$ is projectable, the Cartan tensor and the tensor $\Psi$ vanish and $\nabla^{\mathcal{C}}$ coincides with the canonical connection $\nabla^D$. Furthermore, the Levi-Civita parallelism of $s,\zeta$ implies $\nabla^D_{\mathcal{X}}\sigma=0,\nabla^D_{\mathcal{X}}\psi=0$ for $\mathcal{X}\in\mathcal{H}$ and, on the other hand, we get $\nabla^D_{\mathcal{Y}}\sigma=0,\nabla^D_{\mathcal{Y}}\psi=0$ by evaluating on projectable vector fields. Thus, $\sigma,\psi$ are $\nabla^D$-parallel tensors and $Q=-\sharp_\sigma\circ\flat_\psi$ yields $\nabla^D_{\mathcal{Z}}Q=0$. Correspondingly, (\ref{dmC}), (\ref{torsCartan}) give the following results
\begin{equation}\label{torsinprop}
\begin{array}{l}
\mathcal{D}^{\mathcal{C}}_{\mathcal{Z}}\mathcal{X} =\nabla^D_{\mathcal{Z}}\mathcal{X},\, \mathcal{D}^{\mathcal{C}}_{\mathcal{Z}}\mathcal{Y} =S\circ\nabla^D_{\mathcal{Z}}\circ S'(\mathcal{Y}),\vspace*{2mm}\\
T_{\mathcal{D}^{\mathcal{C}}}(\mathcal{X},\mathcal{X}') =-pr_{\mathcal{V}}[\mathcal{X},\mathcal{X}'],\, T_{\mathcal{D}^{\mathcal{C}}}(\mathcal{Y},\mathcal{Y}') =0,\vspace*{2mm}\\
T_{\mathcal{D}^{\mathcal{C}}}(\mathcal{X},\mathcal{Y}) =S(\nabla^D_{\mathcal{X}}S')(\mathcal{Y}),
\end{array}\end{equation}
where $\mathcal{X}\in\mathcal{H},\mathcal{Y}\in\mathcal{V},\mathcal{Z}\in T\mathcal{T}M$.

$\mathcal{D}^{\mathcal{C}}$ is the Levi-Civita connection of $\mathcal{G}$ iff $T_{\mathcal{D}^{\mathcal{C}}}=0$.
With the torsion formulas (\ref{torsinprop}), we see that $T_{\mathcal{D}^{\mathcal{C}}}(\mathcal{X},\mathcal{X}')=0$ iff $\mathcal{H}$ is integrable.
We claim that $\mathcal{H}$ is integrable iff $M$ is a locally affine manifold and $\mathcal{H}$ is the horizontal bundle associated with the natural, flat, torsionless linear connection of the affine structure of $M$. Indeed, if $M,\mathcal{H}$ are as indicated and if $(x^i)$ are local affine coordinates on $M$, then, we have $X_i=\partial/\partial x^i$, therefore, $\mathcal{H}$ is integrable. Conversely, if $\mathcal{H}$ is integrable, $\mathcal{T}M$ has an atlas of local coordinates $(u^i,v^j)$ such that $\mathcal{H}=span\{\partial/\partial u^i\}$, $\mathcal{V}=span\{\partial/\partial v^i\}$.

Consequently, the local system of equations $u^i=const.$ is equivalent to $x^i=const.$ and $u^i$ are lifts of local coordinates of $M$ with coordinate transformations of the local form $\tilde{u}^i=\tilde{u}^i(u^j)$. With respect to the coordinates $(u^i,\dot{u}^i)$ of $\mathcal{T}M$, where $\dot{u}^i$ are vector coordinates on $M$ with respect to the tangent basis $\partial/\partial u^i$, the coefficients $t_i^j$ of the nonlinear connection $\mathcal{H}$ given by (\ref{Hbases}) vanish. But, the transformation rule (\ref{transft}) tells us that we can have $t_i^j=0,\tilde{t}_i^j=0$ under a coordinate changes $\tilde{u}^i=\tilde{u}^i(u^j)$ iff $\partial\tilde{u}^i/\partial u^j=const.$ Therefore, if $\mathcal{H}$ is integrable, $M$ has a locally affine structure with the local, affine coordinates $(u^i)$ and $\mathcal{H}$ is the horizontal bundle of the corresponding flat, torsionless linear connection.

Finally, for $T_{\mathcal{D}^{\mathcal{C}}}(\mathcal{X},\mathcal{Y})=0$, it suffices to take $\mathcal{X}=\partial/\partial u^i,\, \mathcal{Y}=\partial/\partial\dot{u}^j$. Then, by (\ref{torsinprop}), the condition reduces to $\nabla^D_{\partial/\partial u^i}(\partial/\partial u^j)=0$ and, by the definition of $\nabla^D$, this holds iff $D_{\partial/\partial u^i}(\partial/\partial u^j)=0$, where $D$ is the Levi-Civita connection of $s$ on $M$. Therefore, $D$ has zero curvature and we are done.
\end{proof}
\begin{example}\label{GcazKahler} {\rm If $M$ is a K\"ahler manifold with metric $g$ and K\"ahler form $\Omega$, $\mathcal{T}M$ has a nice generalized Sasaki metric, say $\mathcal{G}_\Omega$,, which satisfies the hypotheses of Proposition \ref{Cartan=LC} and is defined by taking $\sigma=g^h,\psi=\Omega^h$.}\end{example}

For another application, we refer to the Einstein-Cartan version of relativity theory \cite{Tr}. A differentiable manifold $N$ endowed with a (pseudo-)Riemannian metric $g$ and a $g$-metric connection $\nabla$ with torsion will be called an {\it Einstein-Cartan space}, and the metric is an {\it Einstein-Cartan metric}, if it satisfies the following {\it Einstein-Cartan equation}, which is a particular case of the symmetrized Einstein-Cartan-Sciama-Kibble equation (23) of \cite{Tr}
\begin{equation}\label{ECnew} Ric_\nabla(\mathcal{Z},\mathcal{Z}')
+Ric_\nabla(\mathcal{Z}',\mathcal{Z})=2\lambda\mathcal{G} (\mathcal{Z},\mathcal{Z}')
\hspace{3mm} (\mathcal{Z},\mathcal{Z}'\in TN).\end{equation}
In (\ref{ECnew}), $Ric_\nabla$ denotes the (not necessarily symmetric) Ricci curvature of the connection $\nabla$.

We calculate the Ricci tensor of the double metric Cartan connection of a generalized Sasaki metrics $\mathcal{G}$ under the supplementary hypothesis that the horizontal tensor $\beta$ is non degenerate. By (\ref{betacuQ}), this condition is equivalent to $rank(Id-Q^2)=m$, therefore to the fact that $Q^2$ does not have the eigenvalue $1$. From (\ref{relinH}) and (\ref{betacuQ}) we get
$$\beta(\mathcal{X},\mathcal{X})=\sigma(\mathcal{X},\mathcal{X}) +\sigma(Q\mathcal{X},Q\mathcal{X}),$$ hence, if $\sigma$ is positive definite, $\beta$ is non degenerate and positive definite too.

Then, there are local, horizontal, $\beta$-(pseudo-)orthonormal bases $(E_i^\beta)$ $(i=1,...,m)$ and local, horizontal, $\sigma$-(pseudo-)orthonormal bases $(E_i^\sigma)$ and
$(E_i^\beta,SE_i^\sigma)$ are local, tangent bases of $\mathcal{T}M$ where $E_i^\beta$ is not $\mathcal{G}$-perpendicular to $SE_i^\sigma$. The corresponding dual cotangent bases are $(\epsilon^i_\beta,\epsilon^i_\sigma\circ S')$ where $\epsilon^i_\beta,\epsilon^i_\sigma\in ann\,\mathcal{V}$ and $\epsilon^i_\beta(E_j^\beta)=\delta_j^i,\epsilon^i_\sigma(E_j^\sigma)=\delta_j^i$.

The Ricci tensor is the trace of the curvature operator, therefore,
$$ Ric_{\mathcal{D}^\mathcal{C}}(\mathcal{Z}',\mathcal{Z}) = \epsilon^i_\beta(R_{\mathcal{D}^\mathcal{C}}(E_i^\beta,\mathcal{Z})\mathcal{Z}')
+(\epsilon^i_\sigma\circ S')(R_{\mathcal{D}^\mathcal{C}}(SE_i^\sigma,\mathcal{Z})\mathcal{Z}').$$
Furthermore, with (\ref{curbCartan}), we get
\begin{equation}\label{cRicciCartan} \begin{array}{lcl}
Ric_{\mathcal{D}^{\mathcal{C}}}(\mathcal{Y},\mathcal{Z}) &=&\epsilon^i_\sigma(R_{\nabla^{\mathcal{C}}}(SE_i^\sigma, \mathcal{Z})(S'\mathcal{Y})),\vspace*{2mm}\\  Ric_{\mathcal{D}^{\mathcal{C}}}(\mathcal{X},\mathcal{Z}) &=&\epsilon^i_\beta(R_{\nabla^{\mathcal{C}}}(E_i^\beta, \mathcal{Z})\mathcal{X}) +(\epsilon^i_\sigma\circ S')(R_{\nabla^{\mathcal{C}}}(SE_i^\sigma, \mathcal{Z})\mathcal{X})\vspace{2mm}\\ &+& \epsilon^i_\sigma
(R_{\nabla^{\mathcal{C}}}(SE_i^\sigma, \mathcal{Z})(Q\mathcal{X})-QR_{\nabla^{\mathcal{C}}}(SE_i^\sigma, \mathcal{Z})\mathcal{X}),\end{array}\end{equation}
where $\mathcal{X}\in\mathcal{H},\mathcal{Y}\in\mathcal{V}$.
\begin{prop}\label{GEC} A generalized Sasaki metric such that the metric $\sigma$ is projectable to $M$ and $\beta$ is non degenerate is Einstein-Cartan iff $Ric_{{\mathcal{D}^\mathcal{C}}}=0$ and this condition is equivalent to the fact that the projection of $\sigma$ to $M$ is a Ricci flat metric. \end{prop}
\begin{proof} Since $\sigma$ is projectable, we have $\Psi=0$ and (\ref{twocurvatures}) shows that we may replace $R_{\nabla^{\mathcal{C}}}$ by $R_{\nabla^D}$ in (\ref{curbCartan}). The curvature of the canonical connection $\nabla^D$ always satisfies the condition $R_{\nabla^D}(\mathcal{Y},\mathcal{Y}')\mathcal{X}=0$. In the $\sigma$-projectable case, it also satisfies the condition $R_{\nabla^D}(\mathcal{Y},\mathcal{X}')\mathcal{X}=0$ and
$$R_{\nabla^D}(\mathcal{X}',\mathcal{X}'')\mathcal{X}= [R_{D^\sigma}(\pi_*\mathcal{X}',\pi_*\mathcal{X}'')(\pi_*\mathcal{X})]^h,$$ where $D^\sigma$ is the Levi-Civita connection of the projection of $\sigma$ to $M$ (check from the definitions or see the local calculations of \cite{V71}). Then, formulas (\ref{cRicciCartan}) yield
\begin{equation}\label{Ricinprop} Ric_{\mathcal{D}^{\mathcal{C}}}(\mathcal{Y},\mathcal{Z})=0,\,
Ric_{\mathcal{D}^{\mathcal{C}}}(\mathcal{X},\mathcal{Y})=0,\,
Ric_{\mathcal{D}^{\mathcal{C}}}(\mathcal{X}',\mathcal{X})=
[Ric_{D^\sigma}(\mathcal{X}',\mathcal{X})]^h.\end{equation}
(Notice that the trace required in the right hand side of the last equality (\ref{Ricinprop}) is calculated with the non $pr_M\sigma$-(pseudo-)orthonormal basis $E_i^\beta$.)
Since $\mathcal{G}|_{\mathcal{V}}=(1/2)\sigma$, the first equality (\ref{Ricinprop}) shows that the Einstein-Cartan condition (\ref{ECnew}) may hold only for $\lambda=0$. Then, (\ref{ECnew}) necessarily holds for pairs of arguments $(\mathcal{Y}',\mathcal{Y}),(\mathcal{X},\mathcal{Y})$ and it implies the vanishing of the symmetrized tensor $Ric_{\mathcal{D}^{\mathcal{C}}}(\mathcal{X}',\mathcal{X})+ Ric_{\mathcal{D}^{\mathcal{C}}}(\mathcal{X},\mathcal{X}')$. But the third equality (\ref{Ricinprop}) shows that $Ric_{\mathcal{D}^{\mathcal{C}}}(\mathcal{X}',\mathcal{X})$ is symmetric, hence, the Einstein-Cartan condition holds iff $Ric_{\mathcal{D}^{\mathcal{C}}}(\mathcal{X}',\mathcal{X})=0$, equivalently, $pr_M\sigma$ is Ricci flat.
The converse, i.e., that $Ric_{\mathcal{D}^{\mathcal{C}}}=0$ implies (\ref{ECnew}), is obvious; we just have to take $\lambda=0$.
\end{proof}
\begin{rem}\label{bracketconnection} {\rm
Another canonical double metric connection can be obtained
in the following way.
Let $\mathcal{D}^{\mathfrak{g}_{\mathbf{H}}}$ be the Levi-Civita connection of the almost para-Hermitian metric $\mathfrak{g}_{\mathbf{H}}$. Define a new connection $\tilde{\mathcal{D}}^{\mathfrak{g}_{\mathbf{H}}}$ by the formulas
$$
\tilde{\mathcal{D}}^{\mathfrak{g}_{\mathbf{H}}}_{\iota_\pm\mathcal{X}} (\iota_\pm\mathcal{X}')= pr_{V_\pm}\mathcal{D}^{\mathfrak{g}_{\mathbf{H}}}_{\iota_\pm\mathcal{X}} (\iota_\pm\mathcal{X}'),
\, \tilde{\mathcal{D}}^{\mathfrak{g}_{\mathbf{H}}}_{\iota_\pm\mathcal{X}} (\iota_\mp\mathcal{X}')=pr_{V_\mp}[\iota_\pm\mathcal{X}, \iota_\mp\mathcal{X}']_{\mathcal{H}},
$$ 
where $\mathcal{X},\mathcal{X}'\in\mathcal{H}$ and the horizontal metric bracket (\ref{Hbracket}) is transferred to $T\mathcal{T}M$ by $\mathfrak{i}$, i.e.,\footnote{Hereafter, this identification will be used whenever needed without further notice.}
$$[\mathcal{U},\mathcal{U}']_{\mathcal{H}}= \mathfrak{i}[\mathfrak{i}^{-1}\mathcal{U},\mathfrak{i}^{-1}\mathcal{U}']_{\mathcal{H}}, \hspace{2mm}\mathcal{U},\mathcal{U}'\in\chi(\mathcal{T}M).$$
We check that $\tilde{\mathcal{D}}^{\mathfrak{g}_{\mathbf{H}}}$ is a connection by using property (\ref{metriclin}) and the fact that $V_\pm$ are $\mathfrak{g}_{\mathbf{H}}$-orthogonal.

Now, if we take the connection
$$ \mathcal{D}^{\mathfrak{g}_{\mathbf{H}},\mathcal{G}}_{\mathcal{Z}}\mathcal{Z}'= \tilde{\mathcal{D}}^{\mathfrak{g}_{\mathbf{H}}}_{\mathcal{Z}}\mathcal{Z}' +\frac{1}{2}\sharp_{\mathfrak{g}_{\mathbf{H}}} \flat_{(\tilde{\mathcal{D}}^{\mathfrak{g}_{\mathbf{H}}}_{\mathcal{Z}} \mathfrak{g}_{\mathbf{H}})} \mathcal{Z}'\hspace{2mm}(\mathcal{Z},\mathcal{Z}'\in T\mathcal{T}M),$$
a straightforward calculation yields $\mathcal{D}^{\mathfrak{g}_{\mathbf{H}},\mathcal{G}}_{\mathcal{Z}} \mathfrak{g}_{\mathbf{H}}=0$. On the other hand, since $\tilde{\mathcal{D}}^{\mathfrak{g}_{\mathbf{H}}}$ preserves the subbundles $V_\pm$, we have
$$ \mathfrak{g}_{\mathbf{H}}(\sharp_{\mathfrak{g}_{\mathbf{H}}} \flat_{(\tilde{\mathcal{D}}^{\mathfrak{g}_{\mathbf{H}}}_{\mathcal{Z}} \mathfrak{g}_{\mathbf{H}})} \mathcal{Z}', \mathcal{U}) =\tilde{\mathcal{D}}^{\mathfrak{g}_{\mathbf{H}}}_{\mathcal{Z}} \mathfrak{g}_{\mathbf{H}}(\mathcal{Z}',\mathcal{U})=0$$
if $\mathcal{Z}'\in V_\pm,\mathcal{U}\in V_\mp$ ($\mathcal{Z}\in T\mathcal{T}M$).

This shows that the connection $\mathcal{D}^{\mathfrak{g}_{\mathbf{H}},\mathcal{G}}$ also preserves $V_\pm$, hence, it commutes with $H$ and it is a double metric connection.}\end{rem}
\subsection{Generalized Sasaki-K\"ahler metrics}
Now, we are going to relate the metrics with generalized complex structures. The motivation for the material considered in this section is the string theory significance of generalized complex and K\"ahler structures on the big tangent bundle $\mathbf{T}M$, e.g., \cite{Zab}. One may hope that similar structures on the bundle $\pi^{-1}\mathbf{T}M$ and the corresponding transferred structures to $T\mathcal{T}M$ will retain some of that significance.

Let $\mathcal{J}$ be a v.d.g. almost complex structure, represented by the matrix (\ref{J}) where the entries are seen as horizontal tensors, and let $\mathcal{G}$ be a generalized Sasaki metric. By the canonical identification $\mathfrak{i}$ defined by the subordinated, tangent metric $\gamma$ of $\mathcal{G}$, $\mathcal{J}$ goes to $\mathfrak{i}\mathcal{J}\mathfrak{i}^{-1}\mathcal{Z}$ where $\mathfrak{i}^{-1}$ is given by (\ref{ij}) and   $\mathfrak{i}(\mathcal{X},\alpha)=\mathcal{X}+S\sharp_\gamma\alpha$
($\mathcal{X}\in\mathcal{H},\alpha\in ann\,\mathcal{V}$). The result is the following almost complex structure on $\mathcal{T}M$, which we continue to denote by $\mathcal{J}$:
\begin{equation}\label{JpeTrond} \mathcal{J}\mathcal{Z}= Apr_{\mathcal{H}}\mathcal{Z}+\sharp_w\flat_\gamma(S'\mathcal{Z}) +S\sharp_\gamma(\flat_bpr_{\mathcal{H}}\mathcal{Z} -(\flat_\gamma(S'\mathcal{Z}))\circ A). \end{equation}
The structure $\mathcal{J}$ satisfies the $\mathfrak{g}_{\mathbf{H}}$-compatibility condition
\begin{equation}\label{compJGgotic}
\mathfrak{g}_{\mathbf{H}}(\mathcal{J}\mathcal{Z}, \mathcal{Z}')+	 \mathfrak{g}_{\mathbf{H}}(\mathcal{Z}, \mathcal{J}\mathcal{Z}')=0,
\end{equation} equivalently,
\begin{equation}\label{compJGgotic2} \gamma(P_{\mathcal{H}}\mathfrak{J}\mathcal{Z},\mathcal{Z}')+
\gamma(P_{\mathcal{H}}\mathfrak{J}\mathcal{Z}',\mathcal{Z})=0.\end{equation}

Conversely, any almost complex structure $\mathfrak{J}$ of $\mathcal{T}M$ that satisfies (\ref{compJGgotic})
may be identified with a v.d.g. almost complex structure $\mathcal{J}$.
Formula (\ref{JpeTrond}) suggests the following entries of the matrix (\ref{J}) of the required structure $\mathcal{J}$
$$A=pr_{\mathcal{H}}\circ\mathfrak{J},\,\flat_b=\flat_\gamma\circ S'\circ pr_{\mathcal{V}}\circ\mathfrak{J},\, \sharp_w=pr_{\mathcal{H}}\circ\mathfrak{J}\circ S\circ\sharp_\gamma.$$ The endomorphism of $\pi^{-1}(\mathbf{T}M)$ defined by the proposed matrix transfers to
$$\mathfrak{i}\mathcal{J}\mathfrak{i}^{-1}\mathcal{Z}=pr_{\mathcal{H}}\mathfrak{J}\mathcal{Z}
+pr_{\mathcal{V}}\mathfrak{J}pr_{\mathcal{H}}\mathcal{Z} -S\sharp_\gamma[(\flat_\gamma S'pr_{\mathcal{V}}\mathcal{Z})\circ pr_{\mathcal{H}}\circ\mathfrak{J}],$$
where last term is a vertical vector $\Theta$. If we calculate $\gamma(\Theta,\mathcal{Y})$ for $\mathcal{Y}\in\mathcal{V}$ and use the properties of $\gamma,S,S'$ and the compatibility relation
(\ref{compJGgotic2}), we deduce $\Theta=pr_{\mathcal{V}}\mathfrak{J}pr_{\mathcal{V}}\mathfrak{Z}$. Therefore, the total result is exactly $\mathfrak{J}\mathfrak{Z}$ and we are done.

Since $\mathfrak{g}_{\mathbf{H}}$ is a neutral metric, the almost complex structures of $\mathcal{T}M$ that satisfy condition (\ref{compJGgotic}) will be called {\it neutral almost complex structures}. Furthermore, if the pair $(\mathcal{G},\mathcal{J})$ is a v.d.g. almost Hermitian structure, the almost complex structure (\ref{JpeTrond}) is compatible with the generalized Sasakian metric $\mathcal{G}$ and $(\mathcal{G},\mathcal{J})$ is a {\it generalized, Sasakian, almost Hermitian structure}.

We shall transfer the integrability condition of the v.d.g. structure $\mathcal{J}$ to the corresponding neutral structure (\ref{JpeTrond}).
\begin{defin}\label{defJGintegr} {\rm A neutral almost complex structure $\mathcal{J}$ is {\it horizontally integrable}, and the word ``almost" is dropped, if the $\pm i$-eigenbundles of $\mathcal{J}$ are closed under the horizontal, metric bracket $[\,,\,]_{\mathcal{H}}$ (extended by complex linearity).}\end{defin}
\begin{prop}\label{propJGintegr} The neutral almost complex structure $\mathcal{J}$ is integrable iff one of the following equivalent conditions is satisfied:\\
\noindent \hspace*{1mm}(i) $\mathcal{N}^{\mathcal{H}}_{\mathcal{J}}(\mathcal{Z},\mathcal{Z}')
=[\mathcal{J}\mathcal{Z},\mathcal{J}\mathcal{Z}']_{\mathcal{H}}-\mathcal{J} [\mathcal{J}\mathcal{Z},\mathcal{Z}']_{\mathcal{H}} -\mathcal{J}[\mathcal{Z},\mathcal{J}\mathcal{Z}']_{\mathcal{H}} +\mathcal{J}^2[\mathcal{Z},\mathcal{Z}']_{\mathcal{H}}=0,$\vspace*{2mm}\\
\noindent (ii) $\mathcal{N}^{C}_{\tilde{\mathcal{J}}}(\mathfrak{i}^{-1}\mathcal{Z}, \mathfrak{i}^{-1}\mathcal{Z}')
=[\tilde{\mathcal{J}}\mathfrak{i}^{-1}\mathcal{Z},\tilde{\mathcal{J}} \mathfrak{i}^{-1}\mathcal{Z}']_C-\tilde{\mathcal{J}} [\tilde{\mathcal{J}}\mathfrak{i}^{-1}\mathcal{Z},\mathfrak{i}^{-1}\mathcal{Z}']_C \vspace*{2mm}\\ \noindent\hspace*{3cm}-\tilde{\mathcal{J}}[\mathfrak{i}^{-1}\mathcal{Z}, \mathfrak{i}^{-1} \tilde{\mathcal{J}}\mathfrak{i}^{-1}\mathcal{Z}']_C +\tilde{\mathcal{J}}^2 [\mathfrak{i}^{-1}\mathcal{Z},\mathfrak{i}^{-1}\mathcal{Z}']_C\in \mathbf{T}\mathcal{V}$,\vspace*{2mm}\\
\noindent
where $\mathcal{Z},\mathcal{Z}'\in\chi(\mathcal{T}M)$ and the symbol $\mathcal{N}$ is used to remind the Nijenhuis tensor. In (ii), the index $C$ stays for the usual Courant brackets on $\mathcal{T}M$ and, taking into account the decomposition $\mathbf{T}\mathcal{T}M=\mathbf{H}\oplus\mathbf{T}\mathcal{V}$, $\tilde{\mathcal{J}}\in End(\mathbf{T}\mathcal{T}M)$ is defined by
$$\tilde{\mathcal{J}}=\mathcal{J}\circ pr_{\mathbf{H}}+ \mathfrak{v}\circ\mathfrak{h}^{-1}\circ \mathcal{J}\circ\mathfrak{h}\circ\mathfrak{v}^{-1}\circ pr_{\mathbf{T}\mathcal{V}}.$$
\end{prop}
\begin{proof} Obviously, $\tilde{\mathcal{J}}^2=-Id$. On the other hand, for the neutral metric $\mathfrak{g}$ of $\mathbf{T}\mathcal{T}M$ we have $\mathfrak{g}=\mathfrak{g}|_{\mathbf{H}}+\mathfrak{g}|_{\mathbf{T}\mathcal{V}}$ and (\ref{vh-1}) shows that $\mathfrak{h}\circ\mathfrak{v}^{-1}$ transfers the metric $\mathfrak{g}|_{\mathbf{T}\mathcal{V}}$ into $\mathfrak{g}|_{\mathbf{H}}$. Then, since $\mathcal{J}$ is compatible with $\mathfrak{g}_{\mathbf{H}}$, it follows that $\tilde{\mathcal{J}}$ is compatible with $\mathfrak{g}$, hence, $\tilde{\mathcal{J}}$ is a generalized almost complex structure on $\mathbf{T}\mathcal{T}M$ and $\mathcal{N}^{C}_{\tilde{\mathcal{J}}}$ is its Courant-Nijenhuis torsion \cite{Galt}. The equivalence between horizontal integrability of $\mathcal{J}$ and (i) follows by evaluating the {\it horizontal Courant-Nijenhuis torsion} $\mathcal{N}_{\mathcal{J}}^{\mathcal{H}}$ on pairs of $\pm i$-eigenvectors. The equivalence between (i) and (ii) follows from the fact that the Courant bracket on $\mathbf{T}\mathcal{T}M$ is the sum of the horizontal and the vertical brackets. \end{proof}

The horizontal integrability conditions may be expressed in terms of the entries of the matrix (\ref{J}).
\begin{prop}\label{Crainich} The horizontal almost complex structure $\mathcal{J}$ of (\ref{JpeTrond}) is integrable iff it satisfies the conditions
\begin{equation}\label{CrainicH} \begin{array}{l}
pr_{\wedge^3\mathcal{H}}[w,w]=0,\,pr_{\mathcal{H}}\circ R_{(w,A)}=0, \vspace*{2mm}\\
pr_{\mathcal{H}}\{\mathcal{N}_A(\mathcal{X},\mathcal{X}')	 -\sharp_w[(i(\mathcal{X}')i(\mathcal{X})d'b)]\}=0, \vspace*{2mm}\\ d'b_A(\mathcal{X},\mathcal{X}',\mathcal{X}'')= \sum_{Cycl(\mathcal{X},\mathcal{X}',\mathcal{X}'')} d'b(A\mathcal{X},\mathcal{X}',\mathcal{X}''),
\end{array}
\end{equation} where $R$ is the Schouten concomitant, the calculus operations are on $\mathcal{T}M$ and the arguments are horizontal.\end{prop}
\begin{proof} We recall the definition of the Schouten concomitant \cite{V2007}:
$$R_{(w,A)}(X,\alpha)=\sharp_w[L_X(\alpha\circ A)-L_{AX}\alpha]-(L_{\sharp_w\alpha}A)(X).$$
In \cite{Cr}, during the proof of Proposition 2.2, it was shown that, on any manifold $M$, a generalized almost complex structure $\tilde{J}\in End(\mathbf{T}M)$ represented by a matrix (\ref{J}), where the entries  are tensors $\tilde{A},\tilde{w},\tilde{b}$ on $M$, satisfies the following conditions (up to notation)
\begin{equation}\label{prCrainic} \begin{array}{l} <\gamma,pr_{TM}\mathcal{N}^C_{\tilde{J}}((0,\alpha),(0,\beta))>=
[\tilde{w},\tilde{w}](\alpha,\beta,\gamma),\vspace*{2mm}\\ pr_{TM}\mathcal{N}^C_{\tilde{J}}((X,0),(0,,\alpha)) =R_{(\tilde{w},\tilde{A})}(X,\alpha),\vspace*{2mm}\\  pr_{TM}\mathcal{N}^C_{\tilde{J}}((X,0),(Y,0))= \mathcal{N}_{\tilde{A}}(X,Y)- \sharp_{\tilde{w}}[i(Y)i(X)d\tilde{b}],\vspace*{2mm}\\ <pr_{T^*M}\mathcal{N}^C_{\tilde{J}}((X,0),(Y,0)),Z>= d\tilde{b}_{\tilde{A}}(X,Y,Z) \vspace*{2mm}\\ \hspace*{4cm}-\sum_{Cycl(X,Y,Z)}d\tilde{b}(\tilde{A}X,Y,Z).
\end{array}\end{equation} In (\ref{prCrainic}), $\alpha,\beta,\gamma\in\Omega^1(M)$ and $X,Y,Z\in\chi(M)$. Moreover, it was shown that the vanishing of the four projections that enter in formulas (\ref{prCrainic}) implies the vanishing of the Courant-Nijenhuis torsion of $\tilde{J}$.

For the structure $\tilde{\mathcal{J}}$ defined in Proposition \ref{propJGintegr} the manifold is $\mathcal{T}M$ and the tensor fields of (\ref{prCrainic}) are $\tilde{A}=A\oplus A',\tilde{w}=w\oplus w',\tilde{b}=b\oplus b'$. The sum corresponds to $T\mathcal{T}M=\mathcal{H}\oplus\mathcal{V}$ and the prime denotes the transfer from horizontal to vertical by $\mathfrak{h}\circ\mathfrak{v}^{-1}$. The resulting formulas together with (ii), Proposition \ref{propJGintegr}, yield the integrability conditions (\ref{CrainicH}).\end{proof}
\begin{example}\label{exliftJ} {\rm A generalized almost complex structure $J$ on $M$ has a horizontal lift, which defines a neutral almost complex structure $J^h$ on $\mathcal{T}M$. Since $J^h$ is projectable to $M$, the local coordinate expression of the integrability conditions (\ref{CrainicH}) of $J^h$ is the same as that of the integrability conditions of $J$ on $M$. Hence, $J^h$ is horizontally integrable iff $J$ is integrable.}\end{example}
\begin{example}\label{vdcomplex} {\rm Consider a v.d.g. almost complex structure
where the entries $w,b$ of the matrix (\ref{J}) are zero and $A=J$. Then, $J$ is a velocity dependent almost complex structure, which we will see as a horizontal tensor with $J^2=-Id$. This structure is {\it horizontally complex (integrable)} if
\begin{equation}\label{strcomplh} pr_{\mathcal{H}}[J\mathcal{X},J\mathcal{X}']=0
\;\Leftrightarrow\;[\tilde{J}\mathcal{X},\tilde{J}\mathcal{X}'] \in\mathcal{V},
\end{equation} where $X,X'\in\mathcal{H}$ and $\tilde{J}=J\oplus(SJS')$ with $\oplus$ referring to the decomposition $T\mathcal{T}M=\mathcal{H}\oplus\mathcal{V}$. As in the usual case, condition (\ref{strcomplh}) is equivalent with the fact that the $\pm i$-eigenbundles $E_\pm$ of $J$ are closed with respect to the horizontal projection of the Lie bracket. If $J$ is projectable, this means that $J$ is a transversal complex structure of the foliation $\mathcal{V}$. If $J=j^h+(\alpha_iy^i)f^h$, where $j$ is an almost complex structure on $M$, $f$ is a $2$-nilpotent endomorphism of $TM$ that anticommutes with $j$ and $\alpha\in\Omega^1(M)$, then, $J^2=-Id$ and the horizontal integrability condition is equivalent to the system of conditions
$$ 
\begin{array}{l} pr_{\mathcal{H}}\mathcal{N}_{j^h}=0,\,pr_{\mathcal{H}}\mathcal{N}_{f^h}=0,\vspace*{2mm}\\ 

[j^h\mathcal{X},f^h\mathcal{X}']_{\mathcal{H}}+ [f^h\mathcal{X},j^h\mathcal{X}']_{\mathcal{H}} -j^h([f^h\mathcal{X},\mathcal{X}']_{\mathcal{H}}+ [\mathcal{X},f^h\mathcal{X}']_{\mathcal{H}}) 
\vspace*{2mm}\\ 
\hspace*{2cm} 
-f^h([j^h\mathcal{X},\mathcal{X}']_{\mathcal{H}}+[\mathcal{X}, j^h\mathcal{X}']_{\mathcal{H}})=0.
\end{array}$$
}\end{example}

\begin{example}\label{dvsympl} {\rm
A v.d. horizontally symplectic structure is a horizontal, non-degenerate $2$-form $\omega$ such that $d'\omega=0$. As in usual generalized geometry \cite{Galt}, the matrix (\ref{J}) with $A=0,b=\omega,w=\omega^{-1}$ defines a
v.d.g. structure $\mathcal{J}$ that is horizontally integrable.}\end{example}

Finally, we may give the following definition.
\begin{defin}\label{defKgenhor} {\rm The generalized Sasaki, almost Hermitian structure $(\mathcal{G},\mathcal{J})$ is a {\it generalized Sasaki-K\"ahler structure} if the structure $\mathcal{J}$ and its complementary structure $\mathcal{J}'$ are horizontally integrable.}\end{defin}
\begin{example}\label{exgenKhoriz} {\rm Consider the horizontal lift of all the tensors that define a generalized K\"ahler metric of $M$ \cite{Galt} from $M$ to $\mathcal{T}M$. Everything in this lift is projectable and the generalized Sasaki-K\"ahler conditions for the lift coincides with the generalized K\"ahler conditions for the original structure of $M$. Therefore, the horizontal lift is a generalized Sasaki-K\"ahler structure.}\end{example}

{\small Department of Mathematics, University of Haifa, Israel. E-mail: vaisman@math.haifa.ac.il}
\end{document}